\documentclass[a4paper,10pt]{article}
\usepackage[active]{srcltx}
\usepackage{hyperref}
\usepackage[normalem]{ulem}

\usepackage{graphicx}
\usepackage{amssymb}
\usepackage{epstopdf}
\usepackage{amsmath}
\usepackage{amsthm}
\usepackage{amssymb}
\usepackage{latexsym}
\usepackage{mathtools}
\usepackage{mathrsfs}
\usepackage{graphics}
\usepackage{latexsym}
\usepackage{psfrag}
\usepackage{import}
\usepackage{verbatim}
\usepackage{enumerate}
\usepackage{enumitem}
\usepackage{graphicx}
\usepackage[usenames]{color}
\usepackage{cleveref}
\usepackage{orcidlink}

\newcommand{\R}{\mathbb{R}}

\newcommand{\supp}{\text{\rm supp}}

\newcommand{\Id}{\mathbb{I}}

\newcommand{\ve}{\varepsilon}

\newcommand{\T}{\mathcal{T}}

\renewcommand{\L}{\mathcal{L}}

\newcommand{\CD}{\mathsf{CD}}

\newcommand{\Geo}{{\rm Geo}}
\newcommand{\TGeo}{{\rm TGeo}}
\newcommand{\MCP}{\mathsf{MCP}}

\newcommand{\Ent}{{\rm Ent}}

\newcommand{\Dom}{{\rm Dom}}

\newcommand{\fa}{\mathfrak{a}}
\newcommand{\fb}{\mathfrak{b}}

\newcommand{\di}{\mathrm{d}}

\newcommand{\Prob}{\mathcal P}

\newcommand{\Ric}{{\rm Ric}}

\renewcommand{\L}{\mathcal{L}}

\newcommand{\vol}{\mathrm{Vol}}

\newcommand{\TMCP}{\mathsf{TMCP}}
\newcommand{\TCD}{\mathsf{TCD}}

\newcommand{\mm}{\mathfrak m}
\newcommand{\qq}{\mathfrak q}

\newcommand{\ee}{{\rm e}}
\newcommand{\QQ}{\mathfrak Q}

\newcommand{\sfd}{\mathsf d}

\theoremstyle{plain}
\newtheorem{lemma}{Lemma}[section]
\newtheorem{theorem}[lemma]{Theorem}

\newtheorem{proposition}[lemma]{Proposition}

\newtheorem*{theorem*}{Theorem}
\newtheorem*{maintheorem*}{Main Theorem}

\theoremstyle{definition}

\newtheorem{definition}[lemma]{Definition}
\newtheorem*{definition*}{Definition}
\newtheorem{remark}[lemma]{Remark}

\numberwithin{equation}{section}

\providecommand{\keywords}[1]{\small{\textit{Keywords and phrases:} #1}}

\providecommand{\subjclass}[1]{\small{\textit{2020 Mathematics Subject Classification:} #1}}

\title{A singularity theorem in terms of asymptotic  expansion}
\author{Fabio Cavalletti\orcidlink{0000-0003-2494-9029}\thanks{F. Cavalletti: Department of Mathematics, University of Milano (IT), email:fabio.cavalletti@unimi.it.} \ and Andrea Mondino\orcidlink{0000-0002-1932-7148}\thanks{A. Mondino: Mathematical Institute,  University of Oxford  (UK), email: Andrea.Mondino@maths.oxford.ac.uk.}}

\date{\today}     

\begin{document}
\maketitle

\begin{abstract}
We prove a sharp singularity theorem in which the classical focusing
hypothesis of Hawking--Penrose theory is replaced by a condition on the
asymptotic volume growth of timelike geodesics. More precisely, let
$(M,g)$ be a smooth, globally hyperbolic Lorentzian manifold satisfying
the strong energy condition, and let $V\subset M$ be a smooth, compact
Cauchy hypersurface. We associate to the future-directed timelike
geodesics orthogonal to $V$ asymptotic volume-expansion invariants
$\theta_\alpha$, and show that a uniform lower bound
$\theta_\alpha\ge c>0$ forces a sharp upper bound on the time separation
between $V$ and its chronological past. In particular, $(M,g)$ is past
timelike geodesically incomplete.

The estimate is sharp and equality is attained by Lorentzian cones.
Moreover, in the smooth setting we prove a rigidity theorem showing that
equality characterizes these cone models.

The argument extends to globally hyperbolic Lorentzian length spaces
satisfying the synthetic strong energy condition in the form of the
Timelike Measure Contraction Property, $\TMCP^e(0,N)$. In this non-smooth setting we also
obtain corresponding raywise and measure-rigidity statements.
\end{abstract}

\smallskip




\bibliographystyle{plain}

\tableofcontents

\section{Introduction}

The singularity theorems of Penrose~\cite{Penrose65} and
Hawking~\cite{Haw:67} are among the foundational achievements of
mathematical relativity. They demonstrate that spacetime singularities
--- understood as causal geodesic incompleteness --- are not merely
artifacts of highly symmetric solutions of Einstein's equations, but arise
under broad and physically natural assumptions. Penrose's theorem predicts
singularity formation in gravitational collapse, whereas Hawking's theorem
applies to cosmological spacetimes undergoing expansion. Together with
their many subsequent extensions and refinements, these results have
profoundly shaped our understanding of black holes, cosmology, and the
global structure of spacetime.

The underlying mechanism in both theorems is the combination of an energy
condition with the focusing of causal geodesics. Concretely, one assumes
that the spacetime is globally hyperbolic and that the Ricci curvature is
nonnegative along timelike directions in Hawking's theorem and along null
directions in Penrose's theorem. This is supplemented by a local geometric
condition inducing the focusing of causal geodesics, typically formulated
in terms of a mean-curvature bound on a spacelike hypersurface or the
existence of trapped surfaces. Under these hypotheses, one concludes that
the spacetime contains an incomplete causal geodesic. In the classical
framework of general relativity, such geodesic incompleteness is regarded
as the hallmark of a spacetime singularity.

In this paper we establish a singularity theorem of a different nature.
Rather than imposing a local focusing condition on a Cauchy hypersurface,
we assume a lower bound on the \emph{asymptotic volume expansion} of the
timelike geodesic congruence generated by it. Thus the mechanism leading
to incompleteness is encoded in the large-time behaviour of the congruence,
rather than in its infinitesimal behaviour at the initial hypersurface.

A second feature of the result is that its natural framework is the
non-smooth theory of Lorentzian length spaces \cite{KS}. More precisely, the
main theorem is proved under the synthetic strong energy condition $\mathsf{TMCP}^{e}(0,N)$ introduced by the authors in \cite{CaMo:20}, which is rooted in the optimal transport characterizations of timelike Ricci curvature lower bounds established in the smooth setting by McCann \cite{McCann} and by Suhr and the second author \cite{MoSu}.
Recall that $\TMCP^e(0,N)$  is strictly weaker than requiring a
timelike curvature-dimension condition $\TCD^e_p(0,N)$ and is the Lorentzian counterpart
of the classical $\MCP(0,N)$ condition in metric geometry \cite{Ohta1, sturm:II}. The relevance of
this weakening is twofold. On the one hand, the proof of the sharp
singularity estimate ultimately relies only on a one-dimensional
$\MCP(0,N)$ comparison. On the other hand, $\TMCP^e(0,N)$ is satisfied by
important singular Lorentzian models, including generalized Lorentzian
cones \cite{CalistiKettererSaemann}.
\smallskip

For the sake of exposition, let us first describe the result in the smooth
setting. Let $(M,g)$ be an $(n+1)$-dimensional, globally hyperbolic
Lorentzian manifold satisfying the Hawking--Penrose strong energy
condition
\begin{equation}\label{eq:SEC-Intro}
\Ric_g(v,v)\ge 0,
\qquad
\text{for every timelike vector $v\in TM$}.
\end{equation}
Let $V\subset M$ be a smooth, compact Cauchy hypersurface and denote by
$I^\pm(V)$ its chronological future and past. The signed time-separation
from $V$ is defined by
\begin{equation}\label{eq:deftauV-intro}
\tau_V(x):=
\begin{cases}
\sup_{y\in V}\tau(y,x), & x\in I^+(V),\\
-\sup_{y\in V}\tau(x,y), & x\in I^-(V),\\
0, & \text{otherwise}.
\end{cases}
\end{equation}
Away from the cut locus, the integral curves of $-\nabla\tau_V$ are
unit-speed timelike geodesics orthogonal to $V$. More generally, the
localization theory for the time-separation function provides, up to a set
of vanishing volume measure, a decomposition of $I^+(V)\cup I^-(V)$ into
timelike transport rays $X_\alpha$. Parametrizing each ray by
$t=\tau_V$ and imposing $X_\alpha(0)\in V$, the spacetime volume
disintegrates in the smooth case as
\begin{equation}\label{eq:DisintVolvSmooth-Intro}
\vol_g\llcorner_{I^\pm(V)}
=
\int_V \mathfrak m_\alpha\,\vol_{g_V}(\mathrm d\alpha),
\end{equation}
where
\begin{equation}\label{eq:malpha-Intro}
\mathfrak m_\alpha
=
h_\alpha(t)\,\mathcal L^1,
\qquad
h_\alpha(0)=1
\quad
\text{for $\vol_{g_V}$-a.e.\ $\alpha\in V$}.
\end{equation}
The density $h_\alpha$ measures the infinitesimal transverse volume
distortion along the ray $X_\alpha$.

The strong energy condition implies a Bishop--Gromov-type monotonicity
along each ray. In particular,
\[
t\longmapsto
\frac{h_\alpha(t)}{(n+1)t^n}
\]
is non-increasing for $t>0$. This motivates the definition of the
one-dimensional asymptotic area ratio
\begin{equation}\label{eq:defthetaalpha-intro}
\theta_\alpha
:=
\lim_{t\to+\infty}
\frac{h_\alpha(t)}{(n+1)t^n}.
\end{equation}
The quantity $\theta_\alpha$ records the asymptotic transverse volume
expansion along the future-directed ray $X_\alpha$. Our main observation is
that positive asymptotic expansion in the future imposes a sharp upper
bound on how far the same ray can extend into the past.

\subsubsection*{The main result}

The one-dimensional core of the argument is particularly simple. In \Cref{lem:MCP-asymptotic-bound}, we prove that if
$(I,|\cdot|,h\,\mathcal L^1)$ satisfies $\MCP(0,N)$, if
$\sup I=+\infty$, and if
\[
    \theta
    :=
    \lim_{R\to+\infty}\frac{h(R)}{N R^{N-1}}>0,
\]
then
\[
    I\subset
    \left(
    -\left(\frac{h(0)}{N\theta}\right)^{\frac1{N-1}},
    +\infty
    \right).
\]

Combining this estimate with the localization of $\TMCP^e(0,N)$ along the
transport rays of $\tau_V$ gives the following theorem (stated here for simplicity in the smooth setting, for the more general synthetic statement see \Cref{Thm:Secondsingular}).

\begin{theorem}[\Cref{Thm:SecondsingularSmooth}]
Let $(M,g)$ be an $(n+1)$-dimensional, globally hyperbolic, smooth Lorentzian manifold satisfying Penrose-Hawking's strong energy condition \eqref{eq:SEC-Intro}. Let $V\subset M$ be a smooth, compact, Cauchy hypersurface.  Assume moreover that there exists a constant \( c>0 \) such that
\begin{equation}\label{eq:thetaalphageqc-intro}
\theta_\alpha \ge c >0
\qquad
\text{for \(\vol_V\)-a.e.\ \(\alpha \in V\)} .
\end{equation}%
Then 
\begin{equation*}\label{eq:tauVleq1/c-intro}
\sup_{x\in I^{-}(V)} |\tau_{V}(x)|  \leq \left(\frac{1}{(n+1)c}\right)^{\frac{1}{n}}.
\end{equation*}
In particular, $(M,g)$  is not past timelike geodesically complete.

Moreover, for any (possibly non-smooth) past extension $(X,\sfd,\mm, \ll, \leq, \tau)$ of $(M,g)$ which is a timelike non-branching, globally hyperbolic, Lorentzian geodesic space,  satisfying  the $\mathsf{TMCP}^{e}(0,n+1)$ condition, together with its causally-reversed structure, it holds that
\begin{equation}\label{eq:tauVleq1/cCor1-Intro}
\sup_{x\in I^{-}_X(V)} |\tau_{V}(x)|  \leq \left(\frac{1}{(n+1)c}\right)^{\frac{1}{n}},
\end{equation}
where $I^{-}_X(V)\subset X$ is the chronological past of $V$ in $X$.
\\In particular, $(X,\sfd,\mm, \ll, \leq, \tau)$  is not past timelike geodesically complete.
\end{theorem}

The second part of the theorem may be viewed as an inextendibility statement. If a smooth
spacetime satisfying the assumptions above is embedded into a possibly
non-smooth past extension which is timelike non-branching, globally
hyperbolic and satisfies $\TMCP^e(0,n+1)$ in both time orientations, then
the same upper bound \eqref{eq:tauVleq1/cCor1-Intro} holds in the
extension. Hence the singularity cannot be removed by passing to a
lower-regularity extension within this synthetic curvature class.

\subsubsection*{Sharpness and rigidity}

The estimate \eqref{eq:tauVleq1/cCor1-Intro} is sharp. For $N>2$, let
$(Z,\sfd_Z,\mm_Z)$ be a compact
$\CD(-(N-2),N-1)$ metric measure space and consider the Minkowski
$(N-1)$-cone
\[
    \mathcal M^{N-1}(Z)
    =
    {}^{-}[0,+\infty)\times_{\mathrm{id}}^{\,N-1} Z,
    \qquad
    \mathrm d\mm_{\mathcal M}
    =
    r^{N-1}\,\mathrm dr\,\mathrm d\mm_Z.
\]
The results of \cite{CalistiKettererSaemann} imply that this space satisfies
$\TMCP^e(0,N)$. For the slice $$V=\{r_0\}\times Z,$$ the transport rays are
radial and
\[
    h_\alpha(s)
    =
    \left(1+\frac{s}{r_0}\right)^{N-1},
    \qquad
    \theta_\alpha
    =
    \frac1{N r_0^{N-1}}.
\]
Therefore
\[
    \sup_{I^-(V)}|\tau_V|
    =
    r_0
    =
    \left(\frac1{N\theta_\alpha}\right)^{\frac1{N-1}},
\]
so equality is attained.

In the smooth category, these models are rigid. More precisely, assume
that $(M,g)$ is connected, $V$ is connected, and equality holds in
\eqref{eq:tauVleq1/cCor1-Intro}. The asymptotic lower bound first
implies the mean-curvature estimate
\[
    H_V\geq \frac nT,
    \qquad
    T:=\left(\frac1{(n+1)c}\right)^{1/n}.
\]
Equality in the time-separation bound produces a maximal past-directed
$V$-ray of length $T$. A Lorentzian maximal-distance rigidity theorem (see \cite{GrafSplitting}),
combined with the Raychaudhuri--Riccati equation, then propagates equality
to the whole spacetime and yields
\[
    (M,g)
    \simeq
    \left(
    (0,+\infty)\times V,\,
    -\mathrm dr^2+\frac{r^2}{T^2}g_V
    \right).
\]
Thus equality in the smooth theorem characterizes Lorentzian cones; in
particular, if $g_Z:=T^{-2}g_V$, then
\[
    \Ric_{g_Z}\geq -(n-1)g_Z.
\]

In the non-smooth setting of $\TMCP^e(0,N)$ spaces, a meaningful rigidity theory survives at the level of the
conditional measures, see Proposition \ref{prop:raywise-measure-rigidity}. We prove that a maximal transport ray has the exact
model density on its future half,
\[
    h_\alpha(t)
    =
    \left(1+\frac tT\right)^{N-1},
    \qquad t\ge0,
\]
and obtain quantitative bounds on its past density. If maximality holds
for almost every ray and the corresponding sharp past-volume equality is
satisfied, then the full conditional densities are conical. Under an
additional uniqueness/common-tip assumption, inspired by Ohta's
maximal-diameter rigidity theorem for $\MCP(N-1,N)$ spaces
\cite[Theorem~5.5]{OhtaJLMS}, this yields a topological measure-cone
structure.

\subsubsection*{Related literature}

The present work belongs to the rapidly developing theory of synthetic
Lorentzian geometry. The Timelike Curvature--Dimension condition and its
measure-contraction counterpart were introduced in \cite{CaMo:20},
building on the optimal-transport characterizations of timelike Ricci
curvature obtained in the smooth setting by McCann~\cite{McCann} and by
Mondino--Suhr~\cite{MoSu}; see also the survey~\cite{CaMo:22}. Among the
applications of this theory are synthetic versions of Hawking's
singularity theorem, timelike Bishop--Gromov and Bonnet--Myers
inequalities, and Lorentzian isoperimetric-type inequalities
\cite{CaMo:20,cavalletti2024sharpisoperimetrictypeinequalitylorentzian}.

The use of $\TMCP^e$ rather than the stronger $\TCD^e_p$ condition is
particularly natural here. Localization of $\TMCP^e(K,N)$ yields
one-dimensional $\MCP(K,N)$ densities, and the sharp lifetime estimate
proved in this paper is already encoded in this weaker one-dimensional
comparison. This also makes the theorem applicable to generalized
Lorentzian cones constructed in \cite{CalistiKettererSaemann}, which
provide the equality models for our result.

Related developments include inextendibility theorems for Lorentzian
pre-length spaces satisfying synthetic lower bounds on timelike sectional
curvature \cite{GKS:19,AGKS:23,BHS:26}, singularity theorems in the
framework of closed cone structures \cite{Min}, and further comparison
and rigidity results in synthetic Lorentzian geometry; see, among others,
\cite{Braun,Octet,BraunMcCann,Braun:exact,BR:GL}.

A notable feature of the synthetic framework is its ability to accommodate
singular spacetimes lying beyond the scope of classical Lorentzian
geometry. Examples include Penrose's impulsive gravitational waves
\cite{MRS:IPPW} and spacetimes endowed with Lipschitz Lorentzian metrics
whose timelike Ricci curvature is bounded below in the sense of
distributions \cite{BSC:Lipsch}. The present paper contributes to this
program by identifying a new, genuinely asymptotic mechanism for timelike
geodesic incompleteness and by determining its sharp models and rigidity
properties.

\subsection*{Acknowledgments}
For the purpose of Open Access, the authors have applied a CC BY public copyright licence to any Author Accepted Manuscript (AAM) version arising from this submission.

A.\,M.\ acknowledges support from the European Research Council (ERC) under the European Union's Horizon 2020 Research and Innovation Programme, Grant Agreement No.\ 802689, ``CURVATURE''.

Part of this work was carried out during the Thematic Program ``Nonsmooth Riemannian and Lorentzian Geometry'' at the Fields Institute (Toronto), the workshop ``Non-regular Spacetime Geometry'' at the Erwin Schr\"odinger International Institute for Mathematics and Physics (ESI) in Vienna, and the Trimester Program ``Metric Analysis'' at the Hausdorff Institute for Mathematics (HIM) in Bonn.  The authors are grateful to these institutions and to the organisers of the respective programs for their hospitality and for providing inspiring and productive research environments.

\section{Preliminaries}\label{S:Preliminaries}

\subsection{Lorentzian geodesic spaces and synthetic Ricci lower bounds}\label{Ss:Ricci-curvature}

\subsubsection{Some basics on Lorentzian pre-length spaces}\label{SS:BasicsLPLS}

The general framework for the paper is the one of Lorenzian pre-length spaces introduced by Kunzinger-S\"amann in \cite{KS}, inspired by earlier works of Kronheimer-Penrose \cite{CausalSpace} on causal spaces and of Busemann \cite{Busemann} on timelike spaces. An alternative approach put forward by Minguzzi-Suhr  \cite{MiSu:LMP24} derived many topological properties out of basic assumptions on the time-separation. Variants have been studied by several authors with different motivations, let us mention \cite{McCann,  BraunMcCann, ByMiSu:LMP25, Octet, MoSa:25, CaMaMo:25}.
\smallskip

Recall that  $(X,\sfd, \ll, \leq, \tau)$ is a \emph{Lorentzian pre-length space} provided:  $(X,\ll,\leq)$ is a \emph{causal space} (i.e., $\leq$ is a preorder and $\ll$ a transitive relation contained in $\leq$)  additionally  equipped with a proper metric $\sfd$  (i.e., closed and bounded subsets are compact) and a lower semicontinuous function $\tau: X\times X\to [0,\infty]$,  called \emph{time-separation function}, satisfying
\begin{equation}\nonumber
\begin{split}
\tau(x,y)+\tau(y,z)\leq \tau (x,z) &\quad\forall x\leq y\leq z \quad \text{reverse triangle inequality} \\
\tau(x,y)=0, \; \text{if } x\not\leq y, & \quad  \tau(x,y)>0 \Leftrightarrow x\ll y.
\end{split}
\end{equation}
We write $x<y$ whenever $x\le y$ and $x\neq y$. 
We say that $x$ and $y$ are \emph{timelike} (resp.\ \emph{causally}) related if $x\ll y$ (resp.\ $x\le y$).

Let $A\subset X$ be an arbitrary subset. 
The \emph{chronological} (resp.\ \emph{causal}) future of $A$ is defined as
\begin{align*}
I^{+}(A)&:=\{y\in X:\ \exists\, x\in A \text{ such that } x\ll y\},\\ 
\qquad 
J^{+}(A)&:=\{y\in X:\ \exists\, x\in A \text{ such that } x\le y\}.
\end{align*}
Analogously, one defines the \emph{chronological} (resp.\ \emph{causal}) past of $A$.

If $A=\{x\}$ is a singleton, we will slightly abuse notation and write 
$I^{\pm}(x)$ (resp.\ $J^{\pm}(x)$) instead of 
$I^{\pm}(\{x\})$ (resp.\ $J^{\pm}(\{x\})$).

We say that  $(X,\sfd, \ll_r, \leq_r, \tau_r)$ is the \emph{causally-reversed structure} of $(X,\sfd, \ll, \leq, \tau)$ if 
$$
x\ll_r y \iff y\ll x, \quad x\leq_r y\iff y\leq x, \quad \tau_r(x,y)=\tau(y,x).
$$
The set of (resp.\ \emph{timelike}) \emph{geodesics} is defined as: 
\begin{align*}
&\Geo(X):= \{ \gamma\in C([0,1], X)\,:  \, \tau(\gamma_{s}, \gamma_{t})=(t-s)\, \tau(\gamma_{0}, \gamma_{1}),\; \forall s<t\},\\
& \TGeo(X):=\{ \gamma\in \Geo(X) \,:  \,  \tau(\gamma_{0}, \gamma_{1}) > 0\}.   
\end{align*}

\begin{definition}[Timelike non-branching]\label{def:TNB}
A  Lorentzian pre-length space $(X,\sfd, \ll, \leq, \tau)$ is said to be \emph{forward timelike non-branching}  if and only if for any $\gamma^{1},\gamma^{2} \in \TGeo(X)$, it holds:
$$
\exists \;  \bar t\in (0,1) \text{ such that } \ \forall t \in [0, \bar t\,] \quad  \gamma_{ t}^{1} = \gamma_{t}^{2}   
\quad 
\Longrightarrow 
\quad 
\gamma^{1}_{s} = \gamma^{2}_{s}, \quad \forall s \in [0,1].
$$
$X$ is said to be \emph{backward timelike non-branching} if the causally-reversed structure is forward timelike non-branching. In case $X$ is both forward and backward timelike non-branching it is said \emph{timelike non-branching}.
\end{definition}  
Concerning the causal ladder (see \cite{KS,Minguzzi:23} for more details), 
we will only consider \emph{Lorentzian geodesic spaces,} i.e. 
 Lorentzian pre-length spaces $(X,\sfd, \ll, \leq, \tau)$ that  additionally are: 
\begin{itemize}
\item \emph{$\sfd$-Compatible:} every $x\in X$ admits a neighbourhood $U$ and a constant $C$ such that $L_{\sfd}(\gamma)\leq C$ for every causal curve $\gamma$ contained in $U$;
\item  \emph{Geodesic:} for all $x,y\in X$ with $x<y$ there is a future-directed causal curve $\gamma$ from $x$ to $y$ with $\tau(x,y)= L_{\tau}(\gamma)$.
\end{itemize}
We consider the following  version of global hyperbolicity that fits with the previous literature (see \cite[Cor.\;3.8]{Minguzzi:23}): 
a Lorentzian geodesic space $(X,\sfd, \ll, \leq,\tau)$ is called
\begin{itemize}
\item \emph{Causal}: if $\leq$  is also antisymmetric, i.e. $\leq$ is an order; 
\item \emph{Globally hyperbolic}: if it is causal and for every $x,y\in X$ the causal diamond $J^{+}(x)\cap J^{-}(y)$ is compact in $X$.
\end{itemize}

\subsubsection{Timelike optimal transport in a Lorentzian pre-length space}

Let $\mathcal{P}(X)$ be the set of Borel probability measures over $X$.  Denote by $P_i:X\times X\to X$, $i=1,2$ the projection map and let $(P_i)_{\sharp}: \mathcal{P}(X\times X)\to \mathcal{P}(X)$ be the corresponding push-forward map.
\\Given $\mu,\nu\in \mathcal{P}(X)$, consider
\begin{align*}
 \Pi(\mu,\nu)&:=\{\pi\in  \mathcal{P}(X\times X) \,:\, (P_{1})_{\sharp}\pi=\mu, \, (P_{2})_{\sharp}\pi=\nu \}, \nonumber \\
 \Pi_{\leq}(\mu,\nu)&:=\{\pi\in  \Pi(\mu,\nu) \,:\,  \pi(X^{2}_{\leq})=1 \}, 
\nonumber \\
  \Pi_{\ll}(\mu,\nu)&:=\{\pi\in  \Pi(\mu,\nu) \,:\,  \pi(X^{2}_{\ll})=1 \}, 
\end{align*}
where $X^{2}_{\leq}:=\{(x,y) \in X^{2}\,:\, x\leq y \}$ and   $X^{2}_{\ll}:=\{(x,y) \in X^{2}\,:\, x\ll y \}$.


We next recall the notion of $p$-Lorentz-Wasserstein distance, following the convention in \cite{CaMo:20}. The definition below extends to Lorentzian pre-length spaces the corresponding notion given in the smooth Lorentzian setting in \cite{EM17}  (see also \cite{McCann, MoSu}, and \cite{Suhr} for $p=1$). 

\begin{definition}\label{def:Wp}
Let  $(X,\sfd, \ll, \leq, \tau)$ be a Lorentzian pre-length space and let $p\in (0,1]$. Given $\mu,\nu\in \mathcal{P}(X)$, the \emph{$p$-Lorentz-Wasserstein distance} is defined by
\begin{equation}\label{eq:defWp}
\ell_{p}(\mu,\nu):= \sup_{\pi \in \Pi_{\leq}(\mu,\nu)} \left(  \int_{X\times X}  \tau(x,y)^{p} \, \pi(\di x \di y)\right)^{1/p}.
\end{equation}
When $\Pi_{\leq}(\mu,\nu)=\emptyset$ we set $\ell_{p}(\mu,\nu):=-\infty$.
\end{definition}

A coupling  $\pi\in  \Pi_{\leq}(\mu,\nu)$ maximising in \eqref{eq:defWp} is said \emph{$\ell_{p}$-optimal}. The set of \emph{$\ell_{p}$-optimal} couplings from $\mu$ to $\nu$ is denoted by $  \Pi_{\leq}^{p\text{-opt}}(\mu,\nu)$. We also set 
$$
\Pi_{\ll}^{p\text{-opt}}(\mu,\nu):= \Pi_{\leq}^{p\text{-opt}}(\mu,\nu)\cap \Pi_{\ll}(\mu,\nu)
$$
the  family of \emph{timelike}  $\ell_{p}$-optimal couplings.

By gluing of couplings, one can prove that the $p$-Lorentz-Wasserstein distance $\ell_p$ satisfies the reverse triangle inequality on causally related measures (see \cite[Proposition 2.5]{CaMo:20}), thus lifting the Lorentzian metric structure to the space of probability measures.

We conclude this subsection by recalling the notion of $\ell_p$-geodesic $(\mu_s)_{s\in [0,1]}\subset \mathcal{P}(X)$.  The evaluation map is defined by
\begin{equation}\label{def:eet}
\ee_{t}: C([0,1], X) \to X, 
\qquad 
\gamma \mapsto \ee_{t}(\gamma):=\gamma_{t}, 
\qquad \forall t\in [0,1].
\end{equation}

\begin{definition}[$\ell_p$-optimal dynamical plans and $\ell_p$-geodesics]\label{def:ellp-DOP}
Let $(X,\sfd, \ll, \leq, \tau)$ be a Lorentzian pre-length space and let $p\in (0,1]$.  
We say that $\eta\in \mathcal{P}(\Geo(X))$ is an \emph{$\ell_p$-optimal dynamical plan} from $\mu_0\in \mathcal{P}(X)$ to $\mu_1\in \mathcal{P}(X)$ if 
\begin{equation}\label{eq:defODP}
(\ee_0, \ee_1)_{\sharp} \eta 
\in 
\Pi^{p\text{-opt}}_{\leq} \big((\ee_0)_\sharp \eta, (\ee_1)_\sharp \eta\big).
\end{equation}
The collection of all $\ell_p$-optimal dynamical plans from $\mu_{0}$ to $\mu_{1}$ is denoted by ${\rm OptGeo}_{\ell_{p}}(\mu_{0}, \mu_{1})$.
\\ We say that a curve $[0,1]\ni t \mapsto \mu_{t}\in \mathcal{P}(X)$ is an \emph{$\ell_p$-geodesic} if there exists $\eta\in {\rm OptGeo}_{\ell_p}(\mu_0,\mu_1)$ such that
\[
\mu_t=(\ee_t)_\sharp \eta,
\qquad \forall t\in [0,1].
\]
\end{definition}
Observe that if $\eta\in {\rm OptGeo}_{\ell_p}(\mu_0,\mu_1)$, then the associated curve
\[
\mu_t:=(\ee_t)_\sharp \eta, 
\qquad \forall t\in [0,1],
\]
is continuous with respect to the narrow topology and satisfies
\[
\ell_p(\mu_s,\mu_t)
=
(t-s)\,\ell_p(\mu_0,\mu_1),
\qquad 
\forall\, 0\le s\le t\le 1.
\]
Finally, recall that if $X$ is a globally hyperbolic Lorentzian geodesic space, $\mu_0,\mu_1\in \Prob(X)$ have compact support and  $\Pi_{\leq}(\mu_0, \mu_1)\neq \emptyset$, then there always exists an $\ell_p$-optimal dynamical plan $\eta\in {\rm OptGeo}_{\ell_p}(\mu_0,\mu_1)$ (and hence an $\ell_p$-geodesic) connecting $\mu_0$ to $\mu_1$; see \cite[Proposition 2.33]{CaMo:20} for the proof and further properties of $\ell_p$-geodesics.

\subsubsection{Synthetic timelike Ricci lower bounds}

A \emph{measured Lorentzian pre-length space} $(X,\sfd, \mm, \ll, \leq, \tau)$ is a Lorentzian pre-length space endowed with a Radon non-negative measure $\mm$. We say that $(X,\sfd, \mm, \ll, \leq, \tau)$ is globally hyperbolic (resp. geodesic) if $(X,\sfd, \ll, \leq, \tau)$ is so.

\smallskip
We next recall the definition of the Timelike Measure Contraction Property, denoted by $\TMCP^{e}(0,N)$, as introduced in \cite{CaMo:20} (after \cite{McCann} and \cite{MoSu}). Here $K\in \R$ stands for a synthetic lower bound on the timelike Ricci curvature, and $N\in (0,\infty)$ stands for a synthetic upper bound on the dimension. Since in this paper we will only consider the case on non-negative timelike Ricci curvature, we will focus on the case $K=0$. We refer the reader to \cite{CaMo:20} for more details and to \cite{CaMo:22} for a concise overview. 

The definition is based on a suitable concavity property of the entropy functional, which we now briefly recall. Given $\mu \in \mathcal{P}(X)$,
its Boltzmann-Shannon entropy w.r.t. $\mm$ is given by
$$
\Ent(\mu|\mm) = \int_{M} \rho \log\rho \; \mm,
$$
if $\mu = \rho \, \mm$ is absolutely continuous with respect to $\mm$ and $(\rho\log \rho)_{+}$ is  $\mm$-integrable. 
Otherwise we set $\Ent(\mu|\mm) = +\infty$.
We set  $$\Dom(\Ent(\cdot|\mm)):=\{\mu\in \mathcal{P}(X)\,:\, \Ent(\mu|\mm)\in \R\},$$ to be the finiteness domain of the entropy.  Finally, recall the definition of the  Shannon entropy power:
\begin{equation}\label{eq:defSN}
U_{N}(\mu|\mm) : = \exp\left(-\frac{\Ent(\mu|\mm)}{N} \right), \quad \text{for all } \mu\in \Dom(\Ent(\cdot|\mm)).
\end{equation}

\begin{definition}[$\mathsf{TMCP}^{e}(0,N)$ condition, \cite{CaMo:20}]\label{def:TMCP}
Fix $p\in (0,1)$, $K\in \R$, $N\in (0,\infty)$. The measured Lorentzian pre-lengh space $(X,\sfd, \mm, \ll, \leq, \tau)$ satisfies $\mathsf{TMCP}^{e}(0,N)$ if and only if  for any $\mu_{0}\in \Prob_{c}(X)\cap \Dom(\Ent(\cdot|\mm))$ and for any $x_{1}\in X$
such that  $x\ll x_{1}$ for $\mu_{0}$-a.e. $x\in X$, there exists an $\ell_{p}$-geodesic $(\mu_{t})_{t\in [0,1]}$ from $\mu_{0}$ to  $\mu_{1}=\delta_{x_{1}}$
such that 
\begin{equation}\label{eq:defTMCP(0N)}
U_{N}(\mu_{t}|\mm) \geq (1-t)\, U_{N}(\mu_{0}|\mm), \quad \forall t\in [0,1).
\end{equation}
\end{definition}

\begin{remark}
The validity of the $\mathsf{TMCP}^{e}(K,N)$ condition is independent of the choice of $p\in(0,1)$ in Definition~\ref{def:TMCP}. Indeed, for any $p,q\in(0,1)$, a curve $(\mu_t)_{t\in[0,1]}$ with $\mu_1=\delta_{\bar{x}}$ is an $\ell_p$-geodesic if and only if it is an $\ell_q$-geodesic. This follows directly from the definition of $\ell_p$-geodesics (see Definition~\ref{def:ellp-DOP}) and from the observation that the unique coupling between $\mu_0$ and $\mu_1=\delta_{\bar{x}}$ is
$$
\pi=\mu_0\otimes\delta_{\bar{x}},
$$
which is therefore automatically optimal for every $p\in(0,1)$. Hence the class of admissible geodesics entering Definition~\ref{def:TMCP} does not depend on $p$, and neither does the $\mathsf{TMCP}^{e}(K,N)$ condition.
\end{remark}

\subsection{Time separation function from achronal sets}

Another key object in our analysis is the synthetic analogue of the gradient flow lines associated with the time-separation function from a given set $V$. 
We now recall some basic constructions.
\\

\textbf{Standing assumptions.} From now on we will always assume that $(X,\sfd, \ll, \leq, \tau)$ is a globally hyperbolic Lorentzian geodesic space (see subsection \ref{SS:BasicsLPLS} for the definitions).  
 \\

Under such standing assumptions, the time separation $\tau:X\times X\to \R$ is continuous (see \cite[Theorem 3.28]{KS}). A subset $V\subset X$ is called \emph{achronal} if $x\not \ll y$ for every $x,y\in V$. In particular, if $V$ is achronal, then $I^{+}(V)\cap I^{-}(V)= \emptyset$, so we can define the \emph{signed time-separation} to $V$, $\tau_{V}:X\to [-\infty, +\infty]$, by
\begin{equation}\label{eq:deftauV}
\tau_{V}(x):=
\begin{cases}
\sup_{y\in V} \tau(y,x), &\quad \text{ for }x\in I^{+}(V)\\
-\sup_{y\in V} \tau(x,y),& \quad \text{ for }x\in I^{-}(V) \\
0 &\quad \text{ otherwise}
\end{cases}.
\end{equation}
Note that $\tau_{V}$ is lower semi-continuous on $I^{+}(V)$ as supremum of continuous functions, 
and is upper semi-continuous on  $I^{-}(V)$.

To introduce a synthetic analogue of the integral curves of $-\nabla \tau_{V}$, some preparatory tools are needed. 
We briefly recall the construction and refer to \cite[Sec.\;4.1]{CaMo:20} for a more detailed treatment.

\begin{definition}[Future timelike complete (FTC) subsets, \cite{gallowayFTC}]\label{def:FTC}
A subset $V\subset X$ is \emph{future timelike complete} (FTC), if for each point  $x\in I^{+}(V)$, the intersection $J^{-}(x)\cap V \subset V$ has compact closure (w.r.t. $\sfd$) in $V$. Analogously, one defines \emph{past timelike completeness} (PTC). A subset that is both   FTC and PTC is called \emph{timelike complete}.
\end{definition}

\begin{lemma}[Lemma 4.1, \cite{CaMo:20}]\label{L:initialpoint} 
Let $(X,\sfd, \ll, \leq, \tau)$ be a globally hyperbolic Lorentzian geodesic  space and let $V\subset X$ be an achronal FTC (resp. PTC) subset. Then  for each $x\in I^{+}(V)$ (resp. $x\in I^{-}(V)$) there exists a point $y_{x}\in V$ with $\tau_{V}(x)=\tau(y_{x},x)>0$ (resp. $\tau_{V}(x)=-\tau(x,y_{x})<0$).

Moreover for all $x,z\in I^{+}(V)\cup V$,
\begin{equation}\label{eq:tauvzxtau}
\tau_{V}(z) - \tau_{V}(x) \geq \tau(y_{x},z)-\tau(y_{x},x)  \geq  \tau(x,z), 
\end{equation}
provided $(x,z) \in X^{2}_{\leq}$. An analogous statement is valid for 
$x\in I^{-}(V)$.
\end{lemma}

The inequality \eqref{eq:tauvzxtau} can be restated by saying that $\tau_V$ is timelike reverse 1-Lipschitz on $I^+(V)$ and, separately,  
on $I^-(V)$. As shown in \cite{cavalletti2024sharpisoperimetrictypeinequalitylorentzian}, this permits to study the integral lines of maximal steep of $-\tau_V$ on $I^+(V)$ and, separately, on $I^-(V)$ giving a disintegration formula for the reference measure  
$\mm$ whose conditional measures localize the eventual Ricci curvature lower bound to the integral lines of $-\tau_V$. 

For the scope of this note, the interest will be in performing this analysis jointly on 
$I^+(V)$ and $I^-(V)$. 
To start, for ease of writing, we will use the following notation  
$$
 I^{\pm}(V)  : = (I^{+}(V)\cup I^{-}(V)\cup V).
 $$
In general $\tau_V$ fails to be reverse 1-Lipschitz on $I^\pm(V)$: consider for instance 
in Minkowski spacetime $V= \{0\}$, the origin. Let $z$ be any point in the future light cone of $0$ and $x$ be any point in the past light cone of $0$ additionally being in the timelike past of $z$. Then 
$$\tau_V(z)-\tau_V(x)=\tau(0,z)+\tau(x,0)=0<\tau(x,z),$$  contradicting the reverse 1-Lipschitz property. 

Nevertheless, as discussed below,   the reverse 1-Lipschitzianity holds in some special cases of interest. 

\begin{definition}[Empty global edge]
\label{D:GLEdge}
Let $(X,\sfd, \ll, \leq, \tau)$ be a globally hyperbolic Lorentzian geodesic space. 
Let $V\subset X$ be an achronal set. We say that $V$ has \emph{empty global edge} if for any $x \in I^-(V)$ and $z\in I^+(V)$,
any $\gamma\in \TGeo(X)$ connecting $x$ to $z$  intersects $V$. 
\end{definition}

\Cref{D:GLEdge} may be viewed as a global counterpart of the condition that an achronal set has empty edge. Recall that the \emph{edge} of an achronal set \( S \subset X \) consists of those points \( p \in S \) such that every neighbourhood \( U \subset X \) of \( p \) contains points \( p^+ \in I^+(p) \) and \( p^- \in I^-(p) \), together with a timelike curve connecting \( p^- \) to \( p^+ \) that does not intersect \( S \).  Clearly, if $V$ has empty global edge then its edge is empty as well, however the converse implication may fail.

The notion of the edge of an achronal set is classical in Lorentzian geometry and causality theory; see, for instance, the foundational references \cite{PenroseDiffTop,HawEll} or the more recent survey \cite{Ming-SurveyLCT}. 

\begin{remark}[Sufficient conditions for  $V\subset X$ to have empty global edge]
\begin{itemize}
\item Let $(M,g)$ be a spacetime with a continuous Lorentzian metric. If $V$ is a Cauchy hypersurface, then $V$ has empty global edge. For properties of Cauchy hypersurfaces in $C^0$-spacetimes, see \cite{SaC0}.
\item  Recall the definition of future Cauchy development  $D^+(V)$ of $V$: 
\[
D^+(V) : = \{ p \in X \colon \text{every past inext. timelike curve through $p$ intersects $V$}\}.
\]
If $I^+(V) \supset D^+(V)$, then $V$ has empty global edge. 
\end{itemize}
\end{remark}

\begin{proposition}\label{P:1reverseglob}
Let $(X,\sfd, \ll, \leq, \tau)$ be a globally hyperbolic Lorentzian geodesic space and let $V\subset X$ be an achronal, timelike complete subset. 

If $V$ has empty global edge, then
$\tau_V$ is reverse 1-Lipschitz over $I^\pm(V)$.
\end{proposition}

\begin{proof} It suffices to prove that,
for all  $x\leq z$ with $x \in I^-(V)$ and 
$z \in I^+(V)$:
\[
\tau_{V}(z) - \tau_{V}(x) \geq \tau(x,z).
\]
If $\tau(x,z) = 0$, the claim is verified as 
$\tau_{V}(x) \leq 0$ and $\tau_V(z)\geq 0$. 
If $\tau(x,z) > 0$, then there exists  $\gamma\in \TGeo(X)$ with $\gamma_0=x$ and $\gamma_1 = z$. Since $V$  is achronal with empty geodesic  boundary, 
there exists a unique  $s \in (0,1)$ such that 
$\gamma_s \in V$. 
Then 
$\tau(x,z) = \tau(x,\gamma_s) + \tau(\gamma_s,z)
\leq \tau_V(z) - \tau_V(x)$;
by definition indeed
$- \tau_V(x) = \sup_{y\in V} \tau(x,y)$. 
The claim is therefore proved.
\end{proof}

In \cite[Section 6]{cavalletti2024sharpisoperimetrictypeinequalitylorentzian}, the localization of a general reverse 1-Lipschitz function is discussed in detail. 
Hence we now confine ourselves to reporting the main definitions and results.

Given $V\subset X$,  an achronal set with empty global edge, define:
\begin{equation}\label{E:GammaV}
\begin{split}
\Gamma_{V} : = &~ \{ (x,z) \in I^{\pm}(V)^{2} \cap X^{2}_{\leq} \, \colon \, 
 \tau_{V}(z) - \tau_{V}(x)  = \tau(x,z)>0 \} \\ 
 &\qquad \cup \{(x,x) \,:\, x\in I^{\pm}(V)\}.
 \end{split}
\end{equation}
The monotonicity of $\Gamma_{V}$ ensures that pairs of points belonging to $\Gamma_{V}$ are aligned along geodesics: if $(x,z) \in \Gamma_{V}$ with $x \neq z$ and $x \in I^{+}(V)$, then
there exist $y \in V, \gamma \in \TGeo(X)$ and $t \in (0,1)$ such that $y=\gamma_0, \;x = \gamma_{t}, \; z=\gamma_1$, and
$$ \tau(y,\gamma_{s}) = \tau_{V}(\gamma_{s})  \quad \forall s\in [0,1], \qquad
(\gamma_{s},\gamma_{t}) \in \Gamma_{V},  \quad \forall s\in [0,t].
$$
An analogous property holds true if $z \in I^{-}(V)$. Next, define 
$$\Gamma_{V}^{-1}:=\{(x,y)\,:\, (y,x)\in \Gamma_{V}\}$$
 and consider the \emph{transport relation} $R_{V}$ and 
the \emph{transport set with endpoints} $\T_{V}^{end}$
\begin{equation}\label{E:transport}
R_{V} : = \Gamma_{V} \cup \Gamma_{V}^{-1}, \qquad 
\T_{V}^{end} : = P_{1}(R_{V}\setminus \{ x = y \}).
\end{equation}
The transport relation $R_{V}$ will be an equivalence relation on a suitable subset of 
$\T_{V}^{end}$, once initial and final points of the integral lines are removed. The sets
\begin{equation}\label{eq:defendpoints}
\begin{split}
\fa = \fa(\T_{V}^{end}) : =&~ \{ x \in \T_{V}^{end} \colon \nexists y \in \T_{V}^{end} \ s.t. \ (y,x) \in \Gamma_{V}, y\neq x \} \\
\fb = \fb(\T_{V}^{end}) : =&~ \{ x \in \T_{V}^{end} \colon \nexists y \in \T_{V}^{end} \ s.t. \ (x,y) \in \Gamma_{V}, y\neq x \},
\end{split}
\end{equation}
are the \emph{initial} and \emph{final points}, respectively. 
The \emph{transport set without endpoints} is defined by:
\begin{equation}\label{eq:defTV}
\T_{V} : = \T_{V}^{end} \setminus (\fa(\T_{V}^{end}) \cup \fb(\T_{V}^{end})).
\end{equation}

\begin{remark}\label{R:cutlocus}
The set $\fa(\T_{V}^{end})$ shall be understood as the past cut-locus of $V$ while 
$\fb(\T_{V}^{end})$ the future one, being indeed those points, in the smooth scenario, where the geodesics stops to be maximizing. 
\end{remark}

\noindent
If  additionally $V\subset X$ is timelike complete, 
repeating the argument of \cite[Lemma 4.4]{CaMo:20} one can prove that 
$$ 
I^{+}(V) \cup I^{-}(V)  = \mathcal{T}_{V}^{end}    \setminus V.
$$
If $X$ is timelike (both backward and forward) non-branching, then the transport relation $R_{V}$ defines an equivalence relation on $\T_{V}$. 
The corresponding equivalence classes are timelike geodesics that locally maximize the function $\tau_{V}$; we denote them by $X_{\alpha}$.
Moreover, there exists a measurable quotient map 
\[
\QQ : \T_{V} \to Q,
\]
associated with the equivalence relation $R_{V}$ on $\T_{V}$.

The corresponding disintegration of the reference measure $\mm$ and the localization of the Ricci curvature bounds were established in \cite{CaMo:20} (see also \cite{cavalletti2024sharpisoperimetrictypeinequalitylorentzian, BraunMcCann}). We denote by $\mathcal{M}_{+}(X)$ the space of non-negative Radon measures over $X$.

\begin{theorem}\label{T:disintegrationR}
Let  $(X,\sfd, \mm, \ll, \leq, \tau)$ be a globally hyperbolic, timelike non-branching, Lorentzian geodesic  space satisfying $\TMCP^e(0,N)$, 
assume that the causally-reversed structure satisfies the same conditions and 
let $V\subset X$ be a Borel achronal timelike complete subset with empty global edge. \\ Let $\T_{V}^{end}, \fa(\T_{V}^{end}), \fb(\T_{V}^{end})$ and $\T_{V}$ defined in \eqref{E:transport}, \eqref{eq:defendpoints} and \eqref{eq:defTV}.

Then $\mm(\fa(\T_{V}^{end}))=\mm(\fb(\T_{V}^{end})=0$ and the following disintegration formula holds: 
\begin{equation}\label{E:disintegration}
\mm\llcorner_{\T^{end}_{V}} = 
\mm\llcorner_{\T_{V}} 
= \int_{Q} \mm_{\alpha}\, \qq(\di\alpha),
\end{equation}
where $\qq$ is a Borel probability measure over a quotient set $Q \subset \T_{V}$ such that 
$\QQ_{\sharp}( \mm\llcorner_{\T_{V}} ) \ll \qq$ and the map 
$Q \ni \alpha \mapsto \mm_{\alpha} \in \mathcal{M}_{+}(X)$ satisfies the following properties:
\begin{itemize}
\item[(1)] for any $\mm$-measurable set $B\subset X$, the map $\alpha \mapsto \mm_{\alpha}(B)$ is $\qq$-measurable; \smallskip
\item[(2)] for $\qq$-a.e. $\alpha \in Q$, $\mm_{\alpha}$ is concentrated on $\QQ^{-1}(\alpha) = X_{\alpha}$ (strong consistency); \smallskip
\item[(3)] for $\qq$-a.e. $\alpha \in Q$,  
$\mm_{\alpha}\ll  \L^{1}\llcorner_{X_{\alpha}}$;

\item[(4)] for $\qq$-a.e.\,$\alpha \in Q$, the one-dimensional metric measure space 
$(X_{\alpha},|\cdot|, \mm_{\alpha})$ satisfies the classical $\MCP(0,N)$; namely,\;writing $\mm_\alpha=h(\alpha, \cdot)\L^{1}\llcorner_{X_{\alpha}}$, then $h(\alpha, \cdot)$ has an almost everywhere 
representative that is locally Lipschitz and strictly positive 
in the interior of $X_{\alpha}$, continuous on its closure, and satisfying 
\begin{equation}\label{E:MCP0N1d}
\left(\frac{b-\tau_{V}(x_{1})}{b-\tau_{V}(x_{0})}\right)^{N-1}
\leq \frac{h(\alpha, x_{1} ) }{h (\alpha, x_{0})} \leq 
\left( \frac{\tau_{V}(x_{1}) - a }{\tau_{V}(x_{0}) -a)} \right)^{N-1},  
\end{equation}
for all $x_{0},x_{1}\in X_{\alpha}$, with 
$ a< \tau_{V}(x_{0})<\tau_{V}(x_{1})<b$, where $$a:=\inf_{x\in X_\alpha} \tau_V(x), \quad b=\sup_{x\in X_\alpha} \tau_V(x).$$  
\end{itemize}
Moreover, fixed any $\qq$ as above such that $\QQ_{\sharp}( \mm\llcorner_{\T_{V}} ) \ll \qq$, the disintegration is $\qq$-essentially unique in the following sense: if any other 
map $Q \ni \alpha \mapsto \bar \mm_{\alpha} \in \mathcal{P}(X)$
satisfies points (1)-(2), then 
$\bar \mm_{\alpha} = \mm_{\alpha}$ for $\qq$-a.e.  $\alpha \in Q$.
\end{theorem}

\subsection{Future Minkowski content}

Building on \Cref{T:disintegrationR}, in the previous work  \cite{cavalletti2024sharpisoperimetrictypeinequalitylorentzian}, the authors established results concerning the area of achronal subsets of $X$. 
We begin by recalling the definition of the area of an achronal set.

\begin{definition}[Timelike Minkowski content]\label{D:areaachronal}
Let  $A \subset X$ be a Borel achronal set and consider  
the signed time-separation function $\tau_{A}$ from $A$, see \eqref{eq:deftauV}.
We define the \emph{future Minkowski content} of $A$ by 
\begin{equation}\label{E:future}
\begin{split}
    \mm^{+}(A) &: = \inf_{U \in \mathcal{U}} \limsup_{\ve \to 0} \frac{\mm( \tau_{A}^{-1}((0,\ve)) \cap U) }{\ve}, \\ 
\mathcal{U}&: = \{ U \subset X \colon  U \text{ open}, \ A\subset U \}.
\end{split} 
\end{equation}
Analogously, the \emph{past  Minkowski content} of $A$  is defined by 
\begin{equation}\label{E:pastMink}
\begin{split}
    \mm^{-}(A) &: = \inf_{U \in \mathcal{U}} \limsup_{\ve \to 0} \frac{\mm( \tau_{A}^{-1}((-\ve,0)) \cap U) }{\ve}.
\end{split} 
\end{equation}
\end{definition}

\begin{remark}\label{rem:mm+=vol}
If $X$ is a smooth, globally hyperbolic, Lorentzian manifold and $A\subset X$ is a smooth, achronal, future causally complete hypersurface, then $\mm^+(A)$ coincides with the standard area of $A$ computed with respect to the restriction of the ambient Lorentzian metric; see \cite[Remark 4.2]{cavalletti2024sharpisoperimetrictypeinequalitylorentzian}. 

In particular, as shown in \cite{TreudeGrant}, in this smooth setting the time separation function $\tau_A$ is smooth on $I^+(A)\setminus C_+(A)$, where $C_+(A)$ is the future cut-locus of $A$ that is a closed set of zero measure. 
This implies that each level set 
$\tau_A^{-1}(s)$ is a smooth submanifold of 
$I^+(A)\setminus C_+(A)$ having therefore a well-defined volume measure induced by the ambient metric. Finally the coarea formula 
\cite[Proposition 3]{TreudeGrant} gives the claim.
\end{remark}

\section{One-dimensional inequalities}

Leaving the Lorentzian setting for this short section, we will prove some one-dimensional inequalities for $\MCP(0,N)$ spaces, i.e.\ for   $(I,|\cdot|, h \,dx)$, where $I\subset \R$ is an open interval,  $h:I\to (0,\infty)$ is continuous and satisfies \eqref{E:MCP0N1d}.

The Bishop-Gromov inequality (which, in this one-dimensional setting, reduces to an elementary computation) yields that the function $$(0,\infty) \ni R\mapsto  \frac{h(R)}{NR^{N-1}}\searrow$$ is monotone non-increasing. Assuming that $\sup I=\infty$,  we denote 
 \begin{equation}\label{eq:deftheta}
 \theta=\theta_{h,N}:=\lim_{R\to \infty} \frac{h(R)}{NR^{N-1}}=\inf_{R>0} \frac{h(R)}{NR^{N-1}}.
 \end{equation}

 In the next lemma, we relate the asymptotic ``area ratio'' \eqref{eq:deftheta} to the more familiar asymptotic volume ratio,
 $$
 {\rm{AVR}}_N(h):=\lim_{R\to\infty}\frac{1}{R^N}\int_0^R h(t)\,dt
 $$
 which appeared in the recent metric geometry literature (see for instance \cite{BK-MathAnn, CaMa-JEMS}). 

 \begin{lemma}
Let $N\ge 1$ and let $h:[0,\infty)\to[0,\infty)$ be a continuous function such that
\[
t\mapsto \frac{h(t)}{Nt^{\,N-1}}
\]
is decreasing on $(0,\infty)$ and
\[
\lim_{t\to\infty}\frac{h(t)}{Nt^{\,N-1}}=\theta
\]
for some $\theta\ge 0$. Then
\[
\lim_{R\to\infty}\frac{1}{R^N}\int_0^R h(t)\,dt=\theta.
\]
\end{lemma}

\begin{proof}
Set
\[
f(t):=\frac{h(t)}{Nt^{\,N-1}}.
\]
By assumption, $f$ is decreasing on $(0,\infty)$ and $f(t)\to \theta$ as $t\to\infty$.
Fix $\varepsilon>0$. Then there exists $T>0$ such that
\[
\theta \le f(t)\le \theta +\varepsilon, \qquad \text{for all } t\ge T.
\]
For $R>T$, we write
\[
\frac{1}{R^N}\int_0^R h(t)\,dt
=
\frac{1}{R^N}\int_0^T h(t)\,dt
+
\frac{1}{R^N}\int_T^R Nt^{N-1}f(t)\,dt.
\]
The first term in the right hand side tends to $0$ as $R\to\infty$. For the second term, the bounds on $f$ give
\[
\theta \cdot \frac{1}{R^N}\int_T^R t^{N-1}\,dt
\le
\frac{1}{R^N}\int_T^R t^{N-1}f(t)\,dt
\le
(\theta +\varepsilon)\cdot \frac{1}{R^N}\int_T^R t^{N-1}\,dt.
\]
Since
\[
\frac{1}{R^N}\int_T^R Nt^{N-1}\,dt
=
\frac{R^N-T^N}{R^N}\longrightarrow 1
\qquad\text{as }R\to\infty,
\]
it follows that
\[
\lim_{R\to\infty}\frac{1}{R^N}\int_T^R t^{N-1}f(t)\,dt=\theta,
\]
completing the proof.
\end{proof}

\begin{lemma}[The key one-dimensional estimates]\label{lem:MCP-asymptotic-bound}
Let $N\in(1,\infty)$ and let
\[
X=(I,|\cdot|,h\,\mathcal L^1)
\]
be a one-dimensional metric measure space satisfying $\MCP(0,N)$,
where $I\subset\mathbb R$ is an open interval, $0\in I$,
$\sup I=+\infty$, and $h>0$ in $I$.
Then the function
\[
(0,\infty)\ni R\longmapsto
\frac{h(R)}{N R^{N-1}}
\]
is non-increasing. In particular, the limit
\[
\theta=\theta_{h,N}
:=
\lim_{R\to+\infty}\frac{h(R)}{N R^{N-1}}
\]
exists.

Moreover,
\begin{equation}\label{eq:MCP-past-bound}
I\subset
\left(
-\left(\frac{h(0)}{N\theta}\right)^{\frac1{N-1}},
+\infty
\right),
\end{equation}
with the convention that the left endpoint equals $-\infty$ if
$\theta=0$.
\end{lemma}

\begin{proof}
We use the one-dimensional formulation of the $\MCP(0,N)$ condition.
Namely, for every $x_0,x_1\in I$ and every $t\in[0,1]$, one has
\begin{equation}\label{eq:1d-MCP}
h\bigl((1-t)x_0+t x_1\bigr)
\geq
(1-t)^{N-1}h(x_0),
\end{equation}
whenever the segment joining $x_0$ and $x_1$ is contained in $I$.

We first prove the monotonicity of the asymptotic area ratio. Fix
$0<R_1<R_2$ and apply \eqref{eq:1d-MCP} with
\[
x_0=R_2,\qquad x_1=0,
\qquad
t=1-\frac{R_1}{R_2}.
\]
Since
\[
(1-t)R_2+t\,0=R_1,
\]
we obtain
\[
h(R_1)
\geq
\left(\frac{R_1}{R_2}\right)^{N-1}h(R_2).
\]
Equivalently,
\[
\frac{h(R_1)}{R_1^{N-1}}
\geq
\frac{h(R_2)}{R_2^{N-1}}.
\]
It follows that
\[
R\longmapsto\frac{h(R)}{N R^{N-1}}
\]
is non-increasing, and hence the limit defining $\theta$ exists.

Let now $s>0$ be such that $-s\in I$. For any $R>0$, apply
\eqref{eq:1d-MCP} to
\[
x_0=R,\qquad x_1=-s,
\qquad
t=\frac{R}{R+s}.
\]
Since
\[
(1-t)R+t(-s)=0,
\qquad
1-t=\frac{s}{R+s},
\]
we infer that
\begin{equation}\label{eq:MCP-0-R}
h(0)
\geq
\left(\frac{s}{R+s}\right)^{N-1}h(R).
\end{equation}
Rewriting the right-hand side gives
\[
h(0)
\geq
N s^{N-1}
\left(\frac{R}{R+s}\right)^{N-1}
\frac{h(R)}{N R^{N-1}}.
\]
Passing to the limit as $R\to+\infty$ yields
\begin{equation}\label{eq:MCP-theta}
h(0)\geq N\theta\,s^{N-1}.
\end{equation}
If $\theta>0$, this implies
\[
s\leq
\left(\frac{h(0)}{N\theta}\right)^{\frac1{N-1}}.
\]

In fact, the inequality is strict for every $-s\in I$. Indeed, since
$I$ is open, there exists $\varepsilon>0$ such that
$-(s+\varepsilon)\in I$. Applying \eqref{eq:MCP-theta} with
$s+\varepsilon$ in place of $s$ gives
\[
s+\varepsilon
\leq
\left(\frac{h(0)}{N\theta}\right)^{\frac1{N-1}},
\]
and therefore
\[
s<
\left(\frac{h(0)}{N\theta}\right)^{\frac1{N-1}}.
\]
This proves \eqref{eq:MCP-past-bound}. If $\theta=0$, the claim is
void by convention.
\end{proof}

\section{A timelike incompleteness theorem in terms of asymptotic expansion}

Before stating and proving the result, we fix some notation.
Provided $\mm^+(V)<\infty$, it will be convenient to slightly change the normalizations of the conditional probabilities $h(\alpha,\cdot)$ in the following manner.

First, since \( Q \) can be identified with a measurable subset of \( V \), we may, without loss of generality, parametrize each ray \( X_\alpha \) so that
\[
X_\alpha(0)=X_\alpha \cap V .
\]
It then follows that
\[
\int_Q h(\alpha,0)\,\qq(\mathrm d\alpha)
\le \mm^+(V)
< \infty,
\]
and hence \( h(\cdot,0) \in L^1(\qq) \).

For our purposes, it will suffice to consider only those rays \( X_\alpha \) for which \( X_\alpha \cap V \) is a relative interior point of the ray. We denote by \( Q^- \) the set of indices with this property. Hence, since \( 0 \) is an interior point of the chosen parametrization, it follows that for \( \qq \)-a.e.\ \( \alpha \in Q^- \),
\[
h(\alpha,0) > 0.
\]
We may therefore normalize the densities so that
\begin{equation}\label{eq:halpha0=1}
h(\alpha,0)=1
\qquad \text{for } \qq\text{-a.e. } \alpha \in Q^-,
\end{equation}
by correspondingly modifying the quotient measure according to
\begin{equation}\label{eq:halpha(0)=1}
\bar{\qq}
:=
h(\alpha,0)\,\qq\llcorner_{Q^-}
+
\qq\llcorner_{Q\setminus Q^-}.
\end{equation}
Note that \( \bar{\qq} \) is no longer a probability measure. Nevertheless, it satisfies
\begin{equation}\label{eq:barqq(Q)leqmV}
\bar{\qq}(Q) \le \mm^+(V).
\end{equation}
Finally, we assume that the map
\[
\alpha \mapsto \theta_\alpha^{-1}
\]
belongs to \( L^1(\bar{\qq},Q^-) \), where \( \theta_\alpha := \theta_{h_\alpha} \) is defined as in \eqref{eq:deftheta}. 

Observe that the normalization \( h(\alpha,0)=1 \) is also needed in order to fix, once and for all, the scaling of the one-dimensional asymptotic area ratio $\theta_\alpha$.

\begin{theorem}[A synthetic singularity theorem, based on asymptotic volume growth]\label{Thm:Secondsingular}
Let $(X,\sfd,\mm, \ll, \leq, \tau)$ be  a timelike non-branching,  globally hyperbolic, Lorentzian geodesic space. Let $V \subset X$ be an achronal, timelike complete set having finite area, i.e.\ $\mm^{+}(V) <\infty$, 
and  empty global edge.  
Assume moreover that: 
\begin{itemize}
\item Both $(X,\sfd,\mm, \ll, \leq, \tau)$ and the causally-reverse structure satisfy  the $\mathsf{TMCP}^e(0,N)$ condition, for some $N\in (1,\infty)$; 
\item there exists $c>0$ such that, adopting the normalization \eqref{eq:halpha0=1},  it holds 
\begin{equation}\label{eq:thetaalphageqc}
\theta_\alpha \geq c >0 \quad \text{for $\qq$-a.e. $\alpha \in Q^-$}. 
\end{equation}
\end{itemize}
Then 
\begin{equation}\label{eq:tauVleq1/c}
\sup_{x\in I^{-}(V)} |\tau_{V}(x)|  \leq \left(\frac{1}{Nc}\right)^{\frac{1}{N-1}}.
\end{equation}
In particular, $(X,\sfd,\mm, \ll, \leq, \tau)$  is not past timelike geodesically complete. 
\end{theorem}

\begin{proof}
For \( \qq \)-a.e.\ \( \alpha \in Q^- \), the positivity of the one-dimensional asymptotic area ratio \( \theta_\alpha > 0 \) implies that the corresponding ray \( X_\alpha \) has no future endpoint.

Since
\[
\QQ_{\sharp}(\mm\llcorner_{\T_V}) \ll \qq,
\]
it follows that, up to a set of \( \mm \)-measure zero, the set \( I^-(V) \) is foliated by integral curves of \( \tau_V \) having no future endpoint and indexed by elements of \( Q^- \).

Hence, combining the assumption \eqref{eq:thetaalphageqc}  with the bound \eqref{eq:MCP-past-bound}, we obtain that for \( \qq \)-a.e.\ \( \alpha \),
\[
\sup_{x \in X_\alpha \cap I^-(V)} \bigl(-\tau_V(x)\bigr)
\le
\left(\frac{1}{Nc}\right)^{\frac{1}{N-1}} .
\]
The disintegration formula \eqref{E:disintegration} implies that
\[
-\tau_V(x) \le \left(\frac{1}{Nc}\right)^{\frac{1}{N-1}}  , \quad \text{for \( \mm \)-a.e.\ \( x \in I^-(V) \)}.
\]

Since \( -\tau_V \) is lower semicontinuous on \( I^-(V) \) and \( \supp\ \mm = X \) by standing assumption, the above almost-everywhere estimate extends to all of \( I^-(V) \). This concludes the proof of \eqref{eq:tauVleq1/c}.
\\

We finally show that $X$ is not past timelike geodesically complete. 
Consider any timelike geodesic $\gamma$ parametrized by arclength and defined on a maximal (on the left) interval $(-a, 0]\subset (-\infty, 0]$  such that  
$\gamma_{0} \in V$. We claim that 
$$
a\leq  \left(\frac{1}{Nc}\right)^{\frac{1}{N-1}}.
$$
Indeed, if by contradiction for some $s_{0}\in [0,a)$
$$
\tau(\gamma_{s_0}, \gamma_{0}) =s_{0} >  \left(\frac{1}{Nc}\right)^{\frac{1}{N-1}},
$$
the very  definition \eqref{eq:deftauV} of $\tau_{V}$ would imply $$-\tau_{V} (\gamma_{s_{0}})>  \left(\frac{1}{Nc}\right)^{\frac{1}{N-1}},$$ contradicting \eqref{eq:tauVleq1/c}.
\end{proof}

In the case where \( (M,g) \) is an \( (n+1) \)-dimensional globally hyperbolic smooth Lorentzian manifold, and \( V \subset M \) is a smooth, compact, achronal hypersurface with empty global edge, every ray associated with \( \tau_V \) intersects \( V \). This allows one to take \( V \) itself as the quotient set and the induced hypersurface measure \( \vol_V \) as the quotient measure. In this setting, $X_\alpha\cap V$ is an interior point of the ray $X_\alpha$ for every $\alpha$. Thus, we can choose 
\begin{equation}\label{eq:Q-=Vsmooth}
Q^-=V.
\end{equation}
Moreover, the normalization introduced above appears entirely natural and yields the following disintegration formula (which, in this case, corresponds to Fubini-Tonelli's theorem):
\begin{equation}\label{eq:DisintVolvSmooth}
\vol_g\llcorner_{I^{\pm}(V)}=\int_V h(\alpha,\cdot)\, \vol_{g_V}(\di\alpha), \qquad h(\alpha, 0) = 1, \, \vol_V-\text{a.e.} \, \alpha,
\end{equation}
where we have denoted by $\vol_g$ (resp.\ $\vol_{g_V}$) the $(n+1)$-dimensional volume measure of $(M,g)$ (resp.\ the  $n$-dimensional volume measure of $(V,g|_{TV})$).
Below we report the smooth version of \Cref{Thm:Secondsingular}.

\begin{theorem}[A singularity result, based on asymptotic volume growth -- smooth setting] \label{Thm:SecondsingularSmooth}
Let $(M,g)$ be an $(n+1)$-dimensional, globally hyperbolic, smooth Lorentzian manifold satisfying Penrose-Hawking's strong energy condition, i.e.\ $\Ric_g(v,v)\geq 0$ for all $v\in TM$ timelike. Let $V\subset M$ be a smooth, compact, Cauchy hypersurface.  Assume moreover that there exists a constant \( c>0 \) such that
\begin{equation*}
\theta_\alpha \ge c >0
\qquad
\text{for \(\vol_V\)-a.e.\ \(\alpha \in V\)} .
\end{equation*}%
Then 
\begin{equation*}
\sup_{x\in I^{-}(V)} |\tau_{V}(x)|  \leq \left(\frac{1}{(n+1)c}\right)^{\frac{1}{n}}.
\end{equation*}
In particular, $(M,g)$  is not past timelike geodesically complete.

Moreover, for any (possibly non-smooth) past extension $(X,\sfd,\mm, \ll, \leq, \tau)$ of $(M,g)$ which is a timelike non-branching, globally hyperbolic, Lorentzian geodesic space,  satisfying  the $\mathsf{TMCP}^{e}(0,n+1)$ condition, together with its causally-reversed structure, it holds that
\begin{equation*}\label{eq:tauVleq1/cCor1}
\sup_{x\in I^{-}_X(V)} |\tau_{V}(x)|  \leq \left(\frac{1}{(n+1)c}\right)^{\frac{1}{n}},
\end{equation*}
where $I^{-}_X(V)\subset X$ is the chronological past of $V$ in $X$.
\\In particular, $(X,\sfd,\mm, \ll, \leq, \tau)$  is not past timelike geodesically complete.
\end{theorem}

\begin{proof}
Theorem \ref{Thm:SecondsingularSmooth} follows directly from Theorem \ref{Thm:Secondsingular}, once noticing that
\begin{itemize}
\item $(M,g)$ is, in particular, a timelike non-branching, globally hyperbolic, Lorentzian geodesic space satisfying the $\mathsf{TMCP}^{e}(0,n+1)$ condition, together with its causally-reverse structure (see \cite[Theorem A.1]{CaMo:20}, after \cite{McCann, MoSu});

\item since $V$ is a Cauchy hypersurface, it has empty global edge;

\item the smoothness of $V$ implies that all the rays $X_\alpha$, $\alpha\in V$, maximizing the time separation $\tau_V$ from $V$ can be extended across $V$, i.e.\ so that  $X_\alpha(0)=X_\alpha\cap V$ is an interior point of $X_\alpha$;

\item  the disintegration formula $$(\vol_g)|_{I^+(V)}=\int_V h_\alpha\, \vol_V(\di\alpha)$$ corresponds to choosing  $Q=V$ and $\qq=\vol_{g_V}$ in the disintegration theorem, which imply that $h_\alpha(0)=1$ for $\vol_{g_V}$-a.e.\ $\alpha\in V$;

\item The compactness of \( V \) implies that it is timelike complete and has finite \( n \)-dimensional volume. Moreover,
\[
\vol_g^+(V)=\vol_{g_V}(V),
\]
see Remark~\ref{rem:mm+=vol}.
\end{itemize}
The final claim on the non-smooth extension is a consequence of the following observations: if $X$ is a past extension of $(M,g)$, then $(M,g)$ is isomorphically embedded into $X$; moreover, denoting with $I^+_M$ and $I^+_X$ the future sets in $M$ and $X$ respectively, then  $I_M^+(V)=I^+_X(V)$ and $I_M^-(V)\subset I^-_X(V)$. One can then repeat verbatim the proof of Theorem \ref{Thm:Secondsingular} with the aforementioned simplifications.
\end{proof}

\subsection{Sharpness of Theorem \ref{Thm:Secondsingular}}

We briefly recall the notation for Lorentzian generalized cones, following
\cite{CalistiKettererSaemann}. Let $(Z,\sfd_Z,\mm_Z)$ be a metric measure
space and let $I\subset\mathbb R$ be an interval. Given a continuous
warping function $f:I\to[0,\infty)$, the Lorentzian generalized cone
\[
    {}^{-}I\times_f Z
\]
is the Lorentzian warped product with one-dimensional base $I$ endowed
with the negative metric $-dr^2$ and fibre $(Z,\sfd_Z)$. More precisely,
if $\gamma=(r,\beta)$ is an absolutely continuous causal curve, its
Lorentzian length is given by
\[
    L(\gamma)
    =
    \int
    \sqrt{\dot r^{\,2}-f(r)^2|\dot\beta|^2}\,dt,
\]
where $|\dot\beta|$ denotes the metric speed of $\beta$. A causal curve
is future-directed when $\dot r>0$ a.e., and the corresponding
time-separation is obtained by taking the supremum of the Lorentzian
lengths of future-directed causal curves joining two given points.

For $N>2$, the \emph{Minkowski $(N-1)$-cone} over
$(Z,\sfd_Z,\mm_Z)$ is the particular Lorentzian generalized cone
\begin{equation}\label{eq:Minkowski-cone-def}
    \mathcal M^{N-1}(Z)
    :=
    {}^{-}[0,\infty)\times_{\mathrm{id}}^{\,N-1} Z,
\end{equation}
corresponding to the warping function $f(r)=r$. Here the superscript
$N-1$ records the dimensional parameter of the fibre measure. The
canonical reference measure is
\begin{equation}\label{eq:Minkowski-cone-measure}
    \mm_{\mathcal M}
    =
    r^{N-1}\,dr \otimes \mm_Z.
\end{equation}
Since the warping function vanishes at $r=0$, all points of
$\{0\}\times Z$ are identified with a single point
$\mathsf o$, called the \emph{tip} of the cone. Thus, as a set,
\[
    \mathcal M^{N-1}(Z)
    =
    \bigl((0,\infty)\times Z\bigr)\cup\{\mathsf o\}.
\]
On the regular part $(0,\infty)\times Z$, the synthetic construction
is the nonsmooth analogue of the Lorentzian warped metric
\[
    -dr^2+r^2\,\sfd_Z^2.
\]
In particular, the radial curves
\[
    s\longmapsto (r+s,z)
\]
are unit-speed timelike geodesics, as long as $r+s>0$.

\begin{proposition}[Sharpness of Theorem~\ref{Thm:Secondsingular}]
\label{prop:sharpness-lorentzian-cones}
Let $N\in(2,\infty)$ and let $(Z,\sfd_Z,\mm_Z)$ be a compact $\CD(-(N-2),N-1)$ metric measure space with
$\supp\,\mm_Z=Z$.
Consider the Minkowski $(N-1)$-cone $\mathcal M^{N-1}(Z)$ endowed with the measure $\mm_{\mathcal M}$, defined above.  For $r_0>0$, set
\[
    V:=\{r_0\}\times Z.
\]
Then all the assumptions of Theorem~\ref{Thm:Secondsingular} are
satisfied and, adopting the normalization \eqref{eq:halpha0=1},
\begin{equation}\label{eq:theta-cone-sharp}
    \theta_\alpha
    =
    \frac{1}{N r_0^{N-1}}
    \qquad
    \text{for $\bar{\qq}$-a.e.\ $\alpha\in Q^-$}.
\end{equation}
Consequently, choosing
\[
    c:=\frac{1}{N r_0^{N-1}},
\]
equality holds in \eqref{eq:tauVleq1/c}, namely
\begin{equation}\label{eq:sharp-cone-equality}
    \sup_{x\in I^-(V)}|\tau_V(x)|
    =
    r_0
    =
    \left(\frac{1}{Nc}\right)^{\frac{1}{N-1}}.
\end{equation}
In particular, the estimate in
Theorem~\ref{Thm:Secondsingular} is sharp.
\end{proposition}

\begin{proof}
 The fact that the Minkowski cone satisfies
\( \TMCP^e(0,N)\) follows from \cite[Theorem~5.2]{CalistiKettererSaemann} with dimension
$N-1$, warping function
\[
    f(r)=r,
\]
and parameters
\[
    \kappa=0,
    \qquad
    \eta=-1.
\]
Indeed, since $Z$ is non-branching, the cone $X$ is timelike
non-branching by
\cite[Remark~6.2]{CalistiKettererSaemann}. In particular, it is
timelike $p$-essentially non-branching, and hence
\cite[Remark~2.16]{CalistiKettererSaemann} yields the entropic
condition
\(  \TMCP^e(0,N). \)
The same conclusion holds for the causally-reversed structure.
Finally, $X$ is a globally hyperbolic Lorentzian geodesic space; see,
for instance, \cite[Theorem~3.17]{CalistiKettererSaemann}.

We next discuss the slice $V=\{r_0\}\times Z$. It is compact and
achronal, and therefore it is timelike complete. Moreover, it has
empty global edge. Indeed, the radial component of every
future-directed timelike curve is strictly increasing, and hence
every timelike curve joining a point in the chronological past of
$V$ to a point in its chronological future must intersect the level
set $\{r=r_0\}$.

The area of $V$ is finite. More precisely, directly from the
definition of the future Minkowski content,
\begin{align*}
    \mm_{\mathcal M}^+(V)
    &=
    \lim_{\varepsilon\downarrow0}
    \frac{1}{\varepsilon}
    \mm_{\mathcal M}\bigl(\{r_0<r<r_0+\varepsilon\}\bigr)\\
    &=
    \lim_{\varepsilon\downarrow0}
    \frac{\mm_Z(Z)}{\varepsilon}
    \int_{r_0}^{r_0+\varepsilon}r^{N-1}\,dr\\
    &=
    r_0^{N-1}\mm_Z(Z)
    <\infty.
\end{align*}

We now compute the signed time-separation from $V$. If
$\gamma=(r,\beta)$ is any future-directed causal curve, then
\[
    L(\gamma)
    =
    \int
    \sqrt{\dot r^{\,2}-r^2|\dot\beta|^2}\,dt
    \leq
    \int \dot r\,dt.
\]
On the other hand, equality is attained by radial curves, namely
curves with constant $Z$-component. It follows that
\begin{equation}\label{eq:tauV-cone}
    \tau_V((r,z))=r-r_0.
\end{equation}
In particular,
\[
    I^-(V)=\{(r,z):0\leq r<r_0\},
\]
up to the usual identification at the tip, and therefore
\begin{equation}\label{eq:past-distance-cone}
    \sup_{x\in I^-(V)}|\tau_V(x)|
    =
    \sup_{0\leq r<r_0}(r_0-r)
    =
    r_0.
\end{equation}

It remains to compute the one-dimensional asymptotic expansion.
The transport rays associated with $\tau_V$ are precisely the radial
curves
\[
    s\longmapsto (r_0+s,z),
    \qquad
    s\in(-r_0,\infty),
\]
indexed by $z\in Z$. In these coordinates, the reference measure
disintegrates as
\[
    \mm_{\mathcal M}
    =
    (r_0+s)^{N-1}\,ds\otimes\mm_Z.
\]
After imposing the normalization \eqref{eq:halpha0=1}, we may write
\[
    d\bar{\qq}(z)
    :=
    r_0^{N-1}\,d\mm_Z(z),
    \qquad
    h_z(s)
    :=
    \left(1+\frac{s}{r_0}\right)^{N-1},
    \qquad
    s\in(-r_0,\infty).
\]
In particular,
\[
    h_z(0)=1
\]
for every $z\in Z$. Recalling the definition of the
one-dimensional asymptotic area ratio, we obtain
\begin{align*}
    \theta_z
    &=
    \lim_{R\to+\infty}
    \frac{h_z(R)}{N R^{N-1}}\\
    &=
    \lim_{R\to+\infty}
    \frac{(r_0+R)^{N-1}}
         {N r_0^{N-1}R^{N-1}}
    =
    \frac{1}{N r_0^{N-1}}.
\end{align*}
Thus \eqref{eq:thetaalphageqc} holds with equality for
\[
    c=\frac{1}{N r_0^{N-1}}.
\]
Consequently,
\[
    \left(\frac{1}{Nc}\right)^{\frac{1}{N-1}}
    =
    \left(r_0^{N-1}\right)^{\frac{1}{N-1}}
    =
    r_0.
\]
Combining this identity with \eqref{eq:past-distance-cone} gives
\eqref{eq:sharp-cone-equality}, and proves the sharpness of
Theorem~\ref{Thm:Secondsingular}.
\end{proof}

\begin{remark}
The restriction $N>2$ above comes from the range of
\cite[Theorem~5.2]{CalistiKettererSaemann}. Indeed, in order to
construct a cone satisfying $\TMCP^e(0,N)$, that theorem is applied
with fibre dimension parameter $N-1$, which is required to belong to
$(1,\infty)$. For the Minkowski warping function $f(r)=r$, the
corresponding sharp curvature condition on the fibre is
\(
    \CD(-(N-2),N-1).
\)
\end{remark}

\subsection{Rigidity properties}

\begin{theorem}[Geometric rigidity in the  smooth setting]
\label{Thm:secondsingularSmoothRigidity}
Let $(M,g)$ be a connected, $(n+1)$-dimensional, globally
hyperbolic, smooth Lorentzian manifold and let $V\subset M$ be a
smooth, compact, connected Cauchy hypersurface. Assume that
\begin{equation}\label{eq:SEC-smooth-rigidity}
    \Ric_g(v,v)\geq 0
    \qquad\text{for every timelike vector $v\in TM$}.
\end{equation}
Let $h_\alpha$ be the densities in the disintegration
\eqref{eq:DisintVolvSmooth}, normalized by $h_\alpha(0)=1$, and let
$\theta_\alpha$ be the corresponding one-dimensional asymptotic area
 . Suppose that, for some $c>0$,
\begin{equation}\label{eq:theta-lower-rigidity}
    \theta_\alpha\geq c
    \qquad\text{for $\vol_V$-a.e. $\alpha\in V$}.
\end{equation}
Set
\begin{equation}\label{eq:def-T-rigidity}
    T:=\left(\frac{1}{(n+1)c}\right)^{\frac1n}.
\end{equation}
If equality holds in the estimate of
\Cref{Thm:SecondsingularSmooth}, namely
\begin{equation}\label{eq:distance-equality-rigidity}
    \sup_{x\in I^-(V)}|\tau_V(x)|=T,
\end{equation}
then $(M,g)$ is isometric to the Lorentzian cone
\begin{equation}\label{eq:global-cone-rigidity}
    \left((0,+\infty)\times V,
    -\di r^2+\frac{r^2}{T^2}g_V\right),
\end{equation}
where $g_V=g|_{TV}$ and $V$ is identified with the slice
$\{T\}\times V$. In particular, for $$g_Z:=T^{-2}g_V,$$ one has
\begin{equation}\label{eq:fibre-Ricci-rigidity}
    \Ric_{g_Z}\geq -(n-1)g_Z.
\end{equation}
\end{theorem}

\begin{proof}
We divide the proof into three steps.

\emph{Step 1: the asymptotic expansion gives a mean-curvature
bound.}
Let $\nu$ denote the future-pointing unit normal to $V$, and let
$H_V$ be the mean curvature of $V$ with respect to $\nu$, with the
sign convention that the slices of the cone in
\eqref{eq:global-cone-rigidity} have positive mean curvature. On every
future normal ray on which the disintegration is defined, put
\[
    u_\alpha(s):=h_\alpha(s)^{1/n}.
\]
The traced Riccati equation and
\eqref{eq:SEC-smooth-rigidity} imply that $u_\alpha$ is concave.
Moreover,
\begin{equation}\label{eq:ualpha0}
    u_\alpha(0)=1,
    \qquad
    u_\alpha'(0)=\frac{H_V(\alpha)}{n}.
\end{equation}
Since
\begin{equation}\label{eq:ualphaInfty}
    \lim_{R\to+\infty}\frac{u_\alpha(R)}{R}
    =\bigl((n+1)\theta_\alpha\bigr)^{1/n},
\end{equation}
concavity and \eqref{eq:theta-lower-rigidity} yield
\begin{equation}\label{eq:H-lower-rigidity}
    \frac{H_V(\alpha)}{n}
    \geq \bigl((n+1)\theta_\alpha\bigr)^{1/n}
    \geq \frac1T
\end{equation}
for $\vol_V$-a.e. $\alpha\in V$. The smoothness of $V$ and the
continuity of $H_V$ extend the last inequality to every point of
$V$. Hence
\begin{equation}\label{eq:H-global-rigidity}
    H_V\geq \frac nT
    \qquad\text{on $V$}.
\end{equation}

\emph{Step 2: rigidity of the chronological past.}
We first observe that the equality in
\eqref{eq:distance-equality-rigidity} produces a past-directed
$V$-ray of length $T$. Indeed, choose $x_j\in I^-(V)$ such that
$-\tau_V(x_j)\to T$. Global hyperbolicity and the compactness of $V$
give a future-directed maximizing geodesic from $x_j$ to $V$. Read
backwards from $V$, it is a past-directed unit-speed  geodesic
\[
    \gamma_j:[0,s_j]\longrightarrow M,
    \qquad s_j=-\tau_V(x_j)\longrightarrow T,
\]
which maximizes the time separation from $V$ on each of its initial
segments. Writing $\gamma_j(0)=\alpha_j\in V$, compactness gives,
after passing to a subsequence, $\alpha_j\to\alpha\in V$. Continuous
dependence of the geodesic flow then yields a past-directed
unit-speed geodesic
\[
    \gamma:[0,T)\longrightarrow M,
    \qquad -\tau_V(\gamma(t))=t,
\]
where the maximizing property follows by continuity of the time
separation. Thus $\gamma$ is a $V$-ray of maximal length $T$.

Reverse now the time orientation. The timelike Ricci lower bound is
unchanged, while \eqref{eq:H-global-rigidity} becomes the upper bound
$H_V\leq-n/T$. Therefore
\cite[Theorem~5.12]{GrafSplitting} applies with
\[
    \kappa=0,
    \qquad \beta=-\frac nT,
    \qquad b_{\kappa,\beta}=T.
\]
The existence of the maximal $V$-ray obtained above implies
\begin{equation}\label{eq:past-cone-rigidity}
    (I^-(V),g)
    \simeq
    \left((0,T)\times V,
    -\di t^2+\left(1-\frac tT\right)^2g_V\right).
\end{equation}
In particular, the future shape operator of $V$ is
\begin{equation}\label{eq:shape-at-V-rigidity}
    B_0=\frac1T\Id,
    \qquad H_V=\frac nT.
\end{equation}

\emph{Step 3: rigidity of the chronological future.}
Combining \eqref{eq:shape-at-V-rigidity} with \eqref{eq:ualpha0}, \eqref{eq:ualphaInfty} and the concavity of $u_\alpha$, we infer that for $\vol_V$-a.e. $\alpha$ and every $s>0$ in the
corresponding future ray,
\begin{equation}\label{eq:u-squeeze-rigidity}
    1+\bigl((n+1)\theta_\alpha\bigr)^{1/n}s
    \leq u_\alpha(s)
    \leq 1+\frac sT.
\end{equation}
Here the first inequality follows by \eqref{eq:ualphaInfty} and letting the second endpoint tend
to $+\infty$ in the secant-slope inequality for the concave function
$u_\alpha$, while the second one is the tangent-line inequality at
$s=0$ combined with \eqref{eq:ualpha0} and \eqref{eq:shape-at-V-rigidity}. In view of \eqref{eq:theta-lower-rigidity} and
\eqref{eq:def-T-rigidity}, both inequalities in
\eqref{eq:u-squeeze-rigidity} are equalities. Consequently,
\begin{equation}\label{eq:density-equality-rigidity}
    \theta_\alpha=c,
    \qquad
    h_\alpha(s)=\left(1+\frac sT\right)^n
\end{equation}
for $\vol_V$-a.e. $\alpha\in V$.

Equality in the traced Riccati inequality forces equality both in
the Cauchy--Schwarz estimate for the second fundamental form and in
the timelike Ricci bound. Hence the shear vanishes and
\begin{equation}\label{eq:riccati-equality-rigidity}
    B_s=\frac1{T+s}\Id,
    \qquad
    \Ric_g(\dot\gamma_\alpha,\dot\gamma_\alpha)=0.
\end{equation}
It follows that the induced metrics $g_s$ on the regular level sets
of $\tau_V$ satisfy
\[
    \partial_s g_s=2B_s^\flat,
    \qquad
    g_s=\left(\frac{T+s}{T}\right)^2g_V.
\]
Finally, \eqref{eq:density-equality-rigidity} gives the model volume
over every compact subset of $V$. The volume-rigidity statement
\cite[Proposition~6.3]{GrafSplitting} therefore  yields
\begin{equation}\label{eq:future-cone-rigidity}
    (I^+(V),g)
    \simeq
    \left((0,+\infty)\times V,
    -\di s^2+\left(1+\frac sT\right)^2g_V\right).
\end{equation}
Combining \eqref{eq:past-cone-rigidity} and
\eqref{eq:future-cone-rigidity}, and setting $r=T+\tau_V$, proves
\eqref{eq:global-cone-rigidity}.

The warped-product Ricci identities, together with
\eqref{eq:SEC-smooth-rigidity}, imply
$\Ric_{g_Z}\geq-(n-1)g_Z$ for $g_Z=T^{-2}g_V$, proving
\eqref{eq:fibre-Ricci-rigidity}.
\end{proof}

In the next proposition, we establish several rigidity statements in the non-smooth setting. The third statement is inspired by Ohta's rigidity result for the Bonnet--Myers theorem in $\MCP(N-1,N)$ metric measure spaces; see \cite[Theorem~5.5]{OhtaJLMS}.

\begin{proposition}[Raywise and measure rigidity]
\label{prop:raywise-measure-rigidity}
Let $(X,\sfd,\mm,\ll,\leq,\tau)$ and $V\subset X$ satisfy the assumptions
of \Cref{Thm:Secondsingular}. Adopt the normalization
\eqref{eq:halpha0=1}, and assume that
\[
    \theta_\alpha\geq c>0
    \qquad\text{for $\bar{\qq}$-a.e.\ $\alpha\in Q^-$}.
\]
Set
\begin{equation}\label{eq:defT-rigidity}
    T:=\left(\frac{1}{Nc}\right)^{\frac1{N-1}}.
\end{equation}
For $\bar{\qq}$-a.e.\ $\alpha\in Q^-$, write
\[
    X_\alpha\simeq(a_\alpha,+\infty),
    \qquad
    \mm_\alpha=h_\alpha(t)\,\mathcal L^1(dt),
\]
where $t=\tau_V$ and $h_\alpha(0)=1$.

Then the following assertions hold.

\begin{enumerate}
\item[\rm (i)]
Assume that, for some $\alpha\in Q^-$,
\begin{equation}\label{eq:maximal-ray-rigidity}
    a_\alpha=-T.
\end{equation}
Then
\begin{equation}\label{eq:theta-equality-ray}
    \theta_\alpha=c
\end{equation}
and
\begin{equation}\label{eq:future-cone-density}
    h_\alpha(t)
    =
    \left(1+\frac{t}{T}\right)^{N-1},
    \qquad t\in[0,+\infty).
\end{equation}
Moreover, on the past portion of the ray one has
\begin{equation}\label{eq:past-cone-lower-bound}
    \left(1+\frac{t}{T}\right)^{N-1}
    \leq h_\alpha(t)\leq 1,
    \qquad t\in(-T,0].
\end{equation}

\item[\rm (ii)]
Assume that \eqref{eq:maximal-ray-rigidity} holds for
$\bar{\qq}$-a.e.\ $\alpha\in Q^-$ and that
\begin{equation}\label{eq:past-volume-rigidity-assumption}
    \mm(I^-(V))
    =
    \frac{T}{N}\,\bar{\qq}(Q^-).
\end{equation}
Then
\begin{equation}\label{eq:full-cone-density}
    h_\alpha(t)
    =
    \left(1+\frac{t}{T}\right)^{N-1},
    \qquad
    t\in(-T,+\infty),
\end{equation}
for $\bar{\qq}$-a.e.\ $\alpha\in Q^-$.

Consequently, after the change of variable
\[
    r=T+t,
\]
the disintegration of $\mm$ along the transport rays takes the form
\begin{equation}\label{eq:cone-disintegration}
    \mm
    =
    \int_{Q^-}
    r^{N-1}\,dr\,\nu(d\alpha),
    \qquad
    \nu:=T^{-(N-1)}\bar{\qq}\llcorner_{Q^-}.
\end{equation}

\item[\rm (iii)]
Assume, in addition to the assumptions in {\rm (ii)}, the following
Ohta-type uniqueness condition. There exists a point $\mathsf o\in X$
such that:
\begin{itemize}
\item the closure of every ray $X_\alpha$, $\alpha\in Q^-$, has past
endpoint $\mathsf o$;
\item every point of $X\setminus\{\mathsf o\}$ belongs to a unique
maximal transport ray associated with $\tau_V$;
\item the family of such rays is precompact, on every bounded
$t$-interval, in the compact-open topology.
\end{itemize}
Then there exists a topological space $Y$, naturally identified with
the family of transport rays, such that $(X,\mm)$ is a measure cone over
$(Y,\nu)$. More precisely, setting
\[
    \mathcal C(Y)
    :=
    Y\times[0,+\infty)\big/\bigl(Y\times\{0\}\bigr),
\]
there exists a homeomorphism
\[
    \Psi:\mathcal C(Y)\longrightarrow X
\]
sending the tip of $\mathcal C(Y)$ to $\mathsf o$ and satisfying
\begin{equation}\label{eq:measure-cone-rigidity}
    \Psi_\sharp
    \bigl(r^{N-1}\,dr\otimes\nu\bigr)
    =
    \mm.
\end{equation}
Thus $(X,\mm)$ is a cone as a topological measure space.
\end{enumerate}
\end{proposition}

\begin{proof}
We start with the raywise statement. By the localization of the
$\mathsf{TMCP}^e(0,N)$ condition, for $\bar{\qq}$-a.e.\
$\alpha\in Q^-$ the one-dimensional metric measure space
\[
    \bigl((a_\alpha,+\infty),|\cdot|,
    h_\alpha\,\mathcal L^1\bigr)
\]
satisfies $\mathsf{MCP}(0,N)$. In particular, the one-dimensional
$\mathsf{MCP}(0,N)$ comparison yields
\begin{equation}\label{eq:MCP-two-sided-rigidity}
    \left(\frac{b-t_1}{b-t_0}\right)^{N-1}
    \leq
    \frac{h_\alpha(t_1)}{h_\alpha(t_0)}
    \leq
    \left(\frac{t_1-a}{t_0-a}\right)^{N-1},
\end{equation}
whenever
\[
    a<t_0<t_1<b
\]
belong to the ray.

Assume first that $a_\alpha=-T$. Letting $a\downarrow-T$ in the
right-hand inequality of \eqref{eq:MCP-two-sided-rigidity}, we infer
that
\begin{equation}\label{eq:monotone-ratio-cone}
    (-T,+\infty)\ni t
    \longmapsto
    \frac{h_\alpha(t)}{(T+t)^{N-1}}
\end{equation}
is non-increasing. Since $h_\alpha(0)=1$, it follows that
\begin{equation}\label{eq:ratio-at-zero}
    \frac{h_\alpha(0)}{T^{N-1}}
    =
    \frac1{T^{N-1}}
    =
    Nc.
\end{equation}
On the other hand,
\begin{align*}
    \lim_{t\to+\infty}
    \frac{h_\alpha(t)}{(T+t)^{N-1}}
    &=
    \lim_{t\to+\infty}
    \frac{h_\alpha(t)}{t^{N-1}}\\
    &=
    N\theta_\alpha
    \geq Nc.
\end{align*}
Since the function in \eqref{eq:monotone-ratio-cone} is
non-increasing, its limit at $+\infty$ cannot exceed its value at
$t=0$. Hence all the previous inequalities are equalities. In
particular,
\[
    \theta_\alpha=c
\]
and
\[
    \frac{h_\alpha(t)}{(T+t)^{N-1}}
    =
    \frac1{T^{N-1}},
    \qquad t\geq0,
\]
which proves \eqref{eq:future-cone-density}.

For $t\in(-T,0]$, the monotonicity
\eqref{eq:monotone-ratio-cone} gives
\[
    \frac{h_\alpha(t)}{(T+t)^{N-1}}
    \geq
    \frac{h_\alpha(0)}{T^{N-1}},
\]
and therefore
\[
    h_\alpha(t)
    \geq
    \left(1+\frac{t}{T}\right)^{N-1}.
\]
Moreover, letting $b\to+\infty$ in the left-hand inequality of
\eqref{eq:MCP-two-sided-rigidity}, we obtain that $h_\alpha$ is
non-decreasing. Since $h_\alpha(0)=1$, this gives
\[
    h_\alpha(t)\leq1,
    \qquad t\in(-T,0].
\]
This proves \eqref{eq:past-cone-lower-bound} and completes the proof of
{\rm (i)}.

We now prove {\rm (ii)}. By assumption,
$a_\alpha=-T$ for $\bar{\qq}$-a.e.\ $\alpha\in Q^-$. Hence
\eqref{eq:past-cone-lower-bound} gives
\begin{align}
    \int_{-T}^{0}h_\alpha(t)\,dt
    &\geq
    \int_{-T}^{0}
    \left(1+\frac{t}{T}\right)^{N-1}dt \notag\\
    &=
    \frac{T}{N}
    \label{eq:raywise-past-volume-lower}
\end{align}
for $\bar{\qq}$-a.e.\ $\alpha\in Q^-$. Integrating
\eqref{eq:raywise-past-volume-lower} with respect to $\bar{\qq}$ and
using the disintegration formula yields
\begin{equation}\label{eq:global-past-volume-lower}
    \mm(I^-(V))
    \geq
    \frac{T}{N}\bar{\qq}(Q^-).
\end{equation}
Under the equality assumption
\eqref{eq:past-volume-rigidity-assumption}, equality must therefore
hold in \eqref{eq:raywise-past-volume-lower} for
$\bar{\qq}$-a.e.\ $\alpha$. Since
\[
    h_\alpha(t)
    -
    \left(1+\frac{t}{T}\right)^{N-1}
    \geq0
\]
on $(-T,0)$, we conclude that
\[
    h_\alpha(t)
    =
    \left(1+\frac{t}{T}\right)^{N-1}
\]
for $\mathcal L^1$-a.e.\ $t\in(-T,0)$. Recalling that the canonical
representative of an $\mathsf{MCP}(0,N)$ density is continuous in the
interior of its domain, the identity holds for every
$t\in(-T,0]$. Together with \eqref{eq:future-cone-density}, this proves
\eqref{eq:full-cone-density}.

Changing variables according to $r=T+t$, we then have
\[
    h_\alpha(t)\,dt\,\bar{\qq}(d\alpha)
    =
    \frac{r^{N-1}}{T^{N-1}}\,dr\,
    \bar{\qq}(d\alpha)
    =
    r^{N-1}\,dr\,\nu(d\alpha),
\]
which proves \eqref{eq:cone-disintegration}.

We finally prove {\rm (iii)}. Let $Y$ denote the family of maximal
transport rays, endowed with the compact-open topology, and define
\[
    \Psi:Y\times(0,+\infty)\longrightarrow X\setminus\{\mathsf o\}
\]
by
\[
    \Psi(\alpha,r)
    :=
    X_\alpha(r-T).
\]
The common-tip assumption permits to extend $\Psi$ continuously to
$r=0$ by setting
\[
    \Psi(\alpha,0):=\mathsf o
    \qquad\text{for every }\alpha\in Y.
\]
Hence $\Psi$ induces a map
\[
    \Psi:\mathcal C(Y)\longrightarrow X.
\]
The uniqueness assumption implies that this map is bijective.

We claim that it is a homeomorphism. Continuity follows from the
definition of the compact-open topology on the family of rays. To
prove continuity of the inverse, argue as in the proof of
\cite[Lemma~5.4]{OhtaJLMS}. Assume by contradiction that
$x_j\to x$ in $X\setminus\{\mathsf o\}$ while the corresponding rays
$\alpha_j$ do not converge to the ray $\alpha$ containing $x$.
By the assumed precompactness, after passing to a subsequence the
$\alpha_j$ converge on every bounded parameter interval to a maximal
transport ray $\tilde\alpha$. The continuity of $\tau$ and of the
geodesics implies that $\tilde\alpha$ passes through $x$. By the
uniqueness assumption, necessarily $\tilde\alpha=\alpha$, yielding a
contradiction. The same argument at the tip, together with
$\tau_V(\Psi(\alpha,r))=r-T$, gives continuity of the inverse there.
Thus $\Psi$ is a homeomorphism.

It remains to identify the measure. Let $W\subset Y$ be Borel and
$0<a<b<+\infty$. By \eqref{eq:cone-disintegration},
\begin{align*}
    \mm\bigl(\Psi(W\times[a,b])\bigr)
    &=
    \int_W\int_a^b r^{N-1}\,dr\,\nu(d\alpha)\\
    &=
    \nu(W)\int_a^b r^{N-1}\,dr.
\end{align*}
Since sets of this form generate the Borel $\sigma$-algebra of
$\mathcal C(Y)$, a monotone-class argument gives
\[
    \Psi_\sharp
    \bigl(r^{N-1}\,dr\otimes\nu\bigr)
    =
    \mm.
\]
This proves \eqref{eq:measure-cone-rigidity}.
\end{proof}

\begin{remark}\label{rem:MCP-maximal-ray-not-full-rigidity}
The maximal-lifetime condition $a_\alpha=-T$ alone does not force the
full cone density on the past portion of an $\mathsf{MCP}(0,N)$ ray.
For instance,
\[
    h(t):=
    \begin{cases}
        1, & -T<t\leq0,\\[0.2cm]
        \displaystyle
        \left(1+\frac{t}{T}\right)^{N-1},
        & t\geq0,
    \end{cases}
\]
satisfies the one-dimensional $\mathsf{MCP}(0,N)$ comparison
inequalities, is normalized by $h(0)=1$, has
\[
    \theta=\frac{1}{NT^{N-1}},
    \qquad
    a=-T,
\]
but it does not agree with the cone density on $(-T,0)$. Thus an
additional equality condition, such as
\eqref{eq:past-volume-rigidity-assumption}, is necessary in order to
obtain full measure-cone rigidity under $\mathsf{TMCP^e}(0,N)$ alone.
\end{remark}

\textbf{Conflict of interest statement.} The authors have no relevant financial or non-financial interests to disclose. The authors
have no Conflict of interest to declare that are relevant to the content of this article.
\smallskip

\textbf{Data Availability.} Data sharing is not applicable to this article as no new data were created or analyzed in this study.

\footnotesize{
\bibliography{literature.bib}
}
\end{document}